\documentclass{article}
\usepackage[utf8]{inputenc}
\usepackage{amssymb,amsmath,latexsym}
\usepackage{enumerate}
\usepackage{soul}
\usepackage[english]{babel}
\usepackage{amsthm}
\usepackage{verbatim}
\usepackage{graphicx}
\usepackage{color}
\usepackage{tikz}
\usepackage{float}
\usetikzlibrary{calc, intersections, through, backgrounds}
\usetikzlibrary{shapes.geometric}
\tikzstyle{point}=[ball color=white, circle, draw=black, inner sep=0.1cm]
\tikzstyle{red}=[ball color=red, circle, draw=black, inner sep=0.1cm]
\tikzstyle{blue2}=[ball color=blue, circle, draw=black, inner sep=0.1cm]
\tikzstyle{black}=[ball color=black, circle, draw=black, inner sep=0.1cm]
\tikzstyle{green}=[ball color=green, circle, draw=black, inner sep=0.1cm]
\tikzstyle{violet}=[ball color=violet, circle, draw=black, inner sep=0.1cm]
\tikzstyle{yellow}=[ball color=yellow, circle, draw=black, inner sep=0.1cm]
\tikzstyle{white}=[ball color=white, circle, draw=black, inner sep=0.1cm]
\usepackage{caption}
\usepackage{subcaption}
\usepackage[colorinlistoftodos]{todonotes}
\usetikzlibrary{arrows.meta}
\usepackage{authblk}
\usepackage{orcidlink}
\usepackage{mathtools}

\newtheorem{theorem}{Theorem}[section]

\newtheorem{corollary}[theorem]{Corollary}

\newtheorem{definition}[theorem]{Definition}

\newtheorem{lemma}[theorem]{Lemma}

\newtheorem{problem}[theorem]{Problem}
\newtheorem{proposition}[theorem]{Proposition}
\newtheorem{remark}[theorem]{Remark}

\newcommand{\dcap}[2]{\mathrm{cap}^{\times}_{#1}(#2)}
\newcommand{\scap}[2]{\mathrm{cap}^{\boxtimes}_{#1}(#2)}
\newcommand{\ccap}[2]{\mathrm{cap}^{\mathbin{\Box}}_{#1}(#2)}

\DeclareMathOperator{\BoxProduct}{\mathbin{\Box}}
\DeclareMathOperator{\orb}{\mathrm{orb}}

\newcommand{\svSpan}[1]{\sigma^{\boxtimes}_V(#1)}
\newcommand{\dvSpan}[1]{\sigma^{\times}_V(#1)}
\newcommand{\cvSpan}[1]{\sigma^{\BoxProduct}_V(#1)}

\title{Span capacities of graphs}

\author[1,2]{Mateja Grašič\orcidlink{0000-0003-2684-7632}\thanks{mateja.grasic@um.si}}
\author[3]{Christopher Mouron\orcidlink{0009-0008-5122-7519}\thanks{mouronc@rhodes.edu}}
\author[4]{Aljoša Šubašić\orcidlink{0000-0002-6943-0856}\thanks{aljsub@pmfst.hr}}
\author[1,2]{Andrej Taranenko\orcidlink{0000-0002-2438-0496}\thanks{ andrej.taranenko@um.si}}
\author[4]{Tanja Vojković\orcidlink{0000-0002-8160-7757}\thanks{tanja@pmfst.hr}}

\affil[1]{Faculty of Natural Sciences and Mathematics, University of Maribor, Slovenia}
\affil[2]{Institute of Mathematics, Physics and Mechanics, Ljubljana, Slovenia}
\affil[3]{Department of Mathematics and Statistics, Rhodes College, Memphis, TN, 38112 USA}
\affil[4]{Faculty of Science, University of Split, Croatia}

\date{}

\begin{document}

\maketitle

\begin{abstract}
The $d$-capacity of a graph $G$ is introduced as the maximum number of players that can simultaneously traverse $G$ such that each player visits all vertices while maintaining a distance of at least $d$ under various movement rules. We determine their values for paths and cycles and provide bounds for bipartite graphs. Furthermore, we characterize topfull graphs, where the 1-capacities reach their theoretical maximum, establishing a connection to graph factorizations and connectivity.

\end{abstract}

\section{Introduction and motivation}
In 2023, Banič and Taranenko \cite{BaTa23} introduced a graph-theoretic analogue of a well-known topological concept, the span of a continuum. Their work was motivated by the need for maintaining safe interpersonal distances during movement, particularly relevant during the pandemic. In the graph-theoretic setting, they considered two players, Alice and Bob, moving on the vertices of a graph via its edges, each with the aim of visiting all vertices or all edges, giving rise to the notions of vertex spans and edge spans. Three types of movement rules usually used in games on graphs were studied:
\begin{itemize}
\item Traditional movement rules: the players may, independently, either move or stay in place at each step;
\item Active movement rules: both players must move at each step;
\item Lazy movement rules: at each step, exactly one of the players moves.
\end{itemize}
For each of the two spans, these movement rules yield three variants, namely the strong, direct, and Cartesian span, respectively. The terminology is motivated by the characterizations of each span using graph products of the same name. These were also the basis for a polynomial time algorithm for determining the corresponding span for any connected graph. Graphs with the vertex span for each variant equal to zero were also characterised.  

Since then, several aspects of graph spans have been investigated in subsequent papers. 
In \cite{Erceg23}, vertex spans of several graph classes were determined and some relations between vertex spans were presented, i.e. the strong vertex span of any graph is always greater than or equal to the maximum of the values of the direct and Cartesian span, while the direct and Cartesian span differ by at most one, for any observed graph.

It was shown in \cite{DrMiTa25} that the vertex and edge span of the same variant can differ by at most 1. Structural properties of graphs with the strong vertex span equal to $1$ and an exponential time algorithm for the computation of shortest optimal walks with respect to the strong vertex span were also presented.

In \cite{GrMoTa25} using the novel notion of switching walks and triods of a tree the strong vertex span of any tree was determined and a linear time algorithm for its computation was presented. This vastly improved the time complexity of the general algorithm presented in \cite{BaTa23}.

Spans for specific graph classes, including multilayered cycles and paths were studied in \cite{SuVo24}. It was shown that vertex span values depend solely on the length of cycles and paths, instead the number of layers.  
Edge spans were studied in \cite{SuVo25}, and a family of graphs was found for which strong edge span is greater than both direct and Cartesian span. A family with such a property is yet to be found for vertex spans, if it exists, which indicates vertex and edge spans may differ in some important behaviours.

All previous papers consider only two players traversing a graph, which naturally leads to the generalization of the concept of span to settings with more than two players. 
This raises the following question: given that all players moving in a graph must maintain a prescribed distance $d$, what is the maximum number of players that can be placed on a graph? We introduce this number as the \emph{$d$-capacity of a graph}.

This question is also motivated by practical scenarios in which it is important, for safety or other reasons, to determine how many individuals can move through a given space or venue while maintaining a prescribed distance. In analogy with the three movement rules used in the definition of graph spans, we define the strong, direct, and Cartesian $d$-capacity.

The paper is organized as follows. In Section~\ref{sec:definitions}, we establish a necessary theoretical framework and introduce the terminology of marches and tours. Section~\ref{sec:results} contains our main results. We first define the three variants of span capacities and investigate their values for paths and cycles. We then provide bounds for bipartite graphs and conclude with a structural characterization of topfull graphs, linking the maximum $1$-capacity of a graph to its connectivity and factorization properties.  We conclude the paper with some open problems.

\section{Theory of marches}\label{sec:definitions}

First, we provide some initial definitions. All graphs considered in this paper are finite, connected, simple, and undirected. For any graph theoretic notions not explicitly defined here, see \cite{west2001introduction}. 

Let $G$ be a connected graph. Let $\ell$ be a non-negative integer. A \emph{weak walk} $W$ on $G$ is a sequence of vertices $w_0, w_1, \ldots, w_\ell$ in $V(G)$ such that $w_i w_{i+1} \in E(G)$ or $w_i = w_{i+1}$ for all $i\in\{0, 1, \ldots, \ell-1\}$. We denote it by $W: w_0, \ldots, w_\ell$ and also call it a weak $\ell$-walk. Note that if $w_i w_{i+1} \in E(G)$ for every $i\in\{0, 1, \ldots, \ell-1\}$, then this is the common notion of a \emph{walk} in $G$. Thus, every walk is also a weak walk. Note, $\ell$ may be 0, hence a (weak) $0$-walk is just one vertex. Moreover, since we consider (weak) walks parametrized by time, we may refer to indices of the walks as to \emph{points in time}. For this reason we also call a function $f: \{0,1,\ldots, \ell\}\rightarrow V(G)$ with $f(i)f(i+1)\in E(G)$, for any $i\in\{0,1,\ldots, \ell-1\}$, a walk. Similarly, we call a function $f: \{0,1,\ldots, \ell\}\rightarrow V(G)$ with $f(i)=f(i+1)$ or $f(i)f(i+1)\in E(G)$, for any $i\in\{0,1,\ldots, \ell-1\}$, a weak walk. 

\begin{definition}\label{def:distance-tracks}
Let $G$ be a connected graph, $\ell$ a non-negative integer and $f,g:\{0,...,\ell\}\rightarrow V(G)$ functions. The \emph{distance between $f$ and $g$} is defined as
\[
m_G(f,g)= \min\{d_G( f(i), g(i)) \mid i \in \{0,...,\ell\}\}.
\]
\end{definition}

\begin{definition}\label{def:distance-multiple-tracks}
Let $G$ be a connected graph, $p\geq 2$ a natural number and $\ell$ a non-negative integer. Let $f_1, f_2, \ldots, f_p:\{0,\ldots,\ell\}\rightarrow V(G)$ be functions. The distance between $f_1, f_2, \ldots, f_p$ is 
\[
    m_G(f_1, f_2, \ldots, f_p) = \min\{ m_G(f_i, f_j) \mid i,j\in \{1, 2, \ldots, p\} \text{ and } i\not=j  \}.
\]
\end{definition}

\begin{definition}
Let $G$ be a connected graph, $p$ a natural number and $\ell$ a non-negative integer. We say that $F=\{f_1,\ldots,f_p\}$ is a \emph{(weak) $\ell$-march} on $G$ if each $f_i:\{0,1,\ldots,\ell\}\rightarrow V(G)$ is a (weak) walk on $G$. Also, we define \[F(i):=\{f_1(i),\ldots,f_p(i)\} \] to be the \emph{$i$-th stage of $F$.} We call $\ell$ the length of $F$.
\end{definition}

\begin{remark}
When we say that $F=\{f_1,\ldots,f_p\}$ is a (weak) march on $G$, this means that there exists a non-negative integer $\ell$ such that $F$ is a (weak) $\ell$-march on $G$.
\end{remark}

For a (weak) march $F=\{f_1,\ldots,f_p\}$ on some connected graph $G$, we define 
$$m_G(F):=m_G(f_1, f_2, \ldots, f_p).$$ Also, if $p=1$ define $m_G(F):=\infty$.

\begin{definition}
Let $G$ be a connected graph, and $F=\{f_1,\ldots,f_p\}$ a (weak) march on $G$. We say that $F$ is a \emph{collision-free} (weak) march if $m_G(F)\geq 1$.
\end{definition}

\begin{definition}\label{def:orbit}
Let $G$ be a connected graph and let $F=\{f_1,\ldots,f_p\}$ be a collision-free (weak) $\ell$-march on $G$. If $v=f_s(0)\in\{f_1(0),\ldots,f_p(0)\}$, then the
\emph{orbit of $v$ under $F$} is 
\[\orb(F,v)=\orb(f_s,v)=\{f_s(0),f_s(1),\ldots,f_s(\ell)\}.\]
\end{definition} 

Note that in Definition \ref{def:orbit} the orbit is well defined since for any collision-free weak march $F=\{f_1,\ldots,f_p\}$ on the graph $G$ we have that $m_G(F)\geq 1$, hence $v=f_s(0)$ implies that $s$ is unique.

\begin{definition}\label{def:extension}
Let $G$ be a connected graph, $p$ a natural number, and let $F_1=\{f^1_1,\ldots,f^1_p\}$ be a (weak) $\ell_1$-march  on $G$ and $F_2=\{f^2_1,\ldots,f^2_p\}$ be a (weak) $\ell_2$-march on $G$ such that $F_1(\ell_1)=F_2(0)$. An \emph{extension of $F_1$ with $F_2$} is a (weak) $(\ell_1+\ell_2)$-march $F_2F_1=\{f_1,\ldots,f_p\}$ on $G$ defined by 
\[
f_s(i) = \begin{cases}
            f^1_s(i), & \text{if } i\in \{0,\ldots,\ell_1\},\\
            f^2_{k_s}(i-\ell_1), & \text{if } i\in \{\ell_1+1,\ldots,\ell_1+\ell_2\}
         \end{cases},
\]
where $k_s$ is the index such that  $f^1_s(\ell_1)=f^2_{k_s}(0)$. Inductively, we define $F_n F_{n-1} \ldots F_2 F_1=F_n(F_{n-1} \ldots F_2 F_1)$ in the same manner.
\end{definition}

\begin{remark}\label{rem:extension-cf}
	If both $F_1=\{f^1_1,\ldots,f^1_p\}$ and $F_2=\{f^2_1,\ldots,f^2_p\}$ are collision-free (weak) $\ell_1$- and $\ell_2$-marches, respectively, then $F_2F_1$ is also a collision-free (weak) $(\ell_1+\ell_2)$-march.
\end{remark}

\begin{definition}\label{def:reverse-march}
	Let $G$ be a connected graph. Let $F=\{f_1,\ldots,f_p\}$ be a (weak) $\ell$-march on $G$. The \emph{reverse of $F$}, denoted by $F^{-1}=\{\widehat{f}_1,\ldots,\widehat{f}_p\}$, 
	is the (weak) $\ell$-march on $G$ defined by $\widehat{f}_s(i)=f_s(\ell-i)$ for all $s\in \{1,\ldots,p\}$ and $i\in\{0,\ldots,\ell\}$.
\end{definition}

\begin{definition}\label{def:patient-march}
Let $G$ be a connected graph and $F=\{f_1,\ldots,f_p\}$ be a weak $\ell$-march on $G$. We say that $F$ is \emph{patient} if for each $i\in\{0,\ldots,\ell-1\}$ there exists $j\in\{1, 2, \ldots, p\}$ such that $f_j(i)f_j(i+1)\in E(G)$ and $f_{j'}(i)=f_{j'}(i+1)$ for each $j'\in\{1, 2, \ldots, p\}\setminus\{j\}$.
\end{definition}

\begin{remark}
If $F=\{f_1\}$ is a patient $\ell$-march on a connected graph $G$, then Definition \ref{def:patient-march} implies that $f_1$ is a walk on $G$.
\end{remark}

\begin{definition}\label{def:dual-march}
 Let $G$ be a connected graph of order $n$. Suppose that $F=\{f_1,\ldots,f_p\}$ is a patient $\ell$-march on $G$. The \emph{dual of $F$} is a patient $\ell$-march $D_F=\{g_1,\ldots,g_{n-p}\}$ on $G$ such that
 $\{f_1(i),\ldots,f_p(i)\}\cup \{g_1(i),\ldots,g_{n-p}(i)\}=V(G)$, for any $i\in\{0,1,\ldots,\ell\}$. 
\end{definition}

\begin{figure}[htbp]
    \centering
    \begin{tikzpicture} [scale=1]
\draw (1,1)--(1,3)--(5,3) (1,2)--(3,2)--(3,3) (3,2)--(4,1);
\filldraw [black] (1,1) circle (3pt);
\filldraw [white] (4,1) circle (3pt);
\filldraw [black] (1,2) circle (3pt);
\filldraw [black] (2,2) circle (3pt);
\filldraw [black] (3,2) circle (3pt);
\filldraw [black] (1,3) circle (3pt);
\filldraw [black] (2,3) circle (3pt);
\filldraw [white] (3,3) circle (3pt);
\filldraw [white] (4,3) circle (3pt);
\filldraw [white] (5,3) circle (3pt);
\end{tikzpicture}
\hspace{50 pt}
\begin{tikzpicture} [scale=1]
\draw (1,1)--(1,3)--(5,3) (1,2)--(3,2)--(3,3) (3,2)--(4,1);
\filldraw [black] (1,1) circle (3pt);
\filldraw [black] (4,1) circle (3pt);
\filldraw [black] (1,2) circle (3pt);
\filldraw [black] (2,2) circle (3pt);
\filldraw [white] (3,2) circle (3pt);
\filldraw [black] (1,3) circle (3pt);
\filldraw [black] (2,3) circle (3pt);
\filldraw [white] (3,3) circle (3pt);
\filldraw [white] (4,3) circle (3pt);
\filldraw [white] (5,3) circle (3pt);
\end{tikzpicture}
    \caption{An example of two consecutive stages of a patient march.}
    \label{fig:figure_1}
\end{figure}
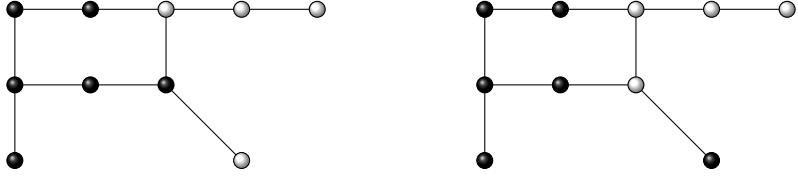

With respect to Definition \ref{def:dual-march}, we think of $\{f_1(i),\ldots,f_p(i)\}$ as the \emph{occupied vertices} and $\{g_1(i),\ldots,g_{n-p}(i)\}$ as the \emph{vacant vertices} of $G$ at stage $i$. Note that this implies that $f_s(i)\not=g_t(i)$ for all applicable $i,s,t$. An example of a patient march $F$ and its dual $D_F$, say in their initial stage (left) and after the first move (right), is given in Figure \ref{fig:figure_1}, where at each stage the occupied vertices are shown in black and the vacant vertices are white.  

Also note that if a patient march $F$  on $G$ is defined, then its dual $D_F$ is automatically defined. Also, $F=D_{D_F}$ is the dual of $D_F$, so it follows that if $D_F$ is defined, then $F$ is also defined. 

Next, we introduce the notion of (weak) tracks, which are (weak) walks through all the vertices of a given graph. These were already defined in \cite{Erceg23}, but we state them here for clarity. We also define a new notion of patient $\ell$-tracks, needed for the definition of the Cartesian $d$-capacity.

\begin{definition} Let $G$ be a connected graph and $\ell$ a non-negative integer. We say that a (weak) $\ell$-walk $f : \{0,...,\ell\} \rightarrow V(G)$ is a \emph{(weak) $\ell$-track on $G$} if $f$ is surjective.
\end{definition}

\begin{definition}
    Let $G$ be a connected graph and let $F=\{f_1,\ldots,f_p\}$ be a (weak / patient) $\ell$-march on $G$. If $f_i$ is a (weak) $\ell$-track for all $i\in \{1,\ldots,p\}$ we say that $F$ is a \emph{(weak / patient) $\ell$-tour on $G$}.
\end{definition}
   
\begin{remark}
Let $G$ be a connected graph. Where the order of the domain is not important for the context, we say that a function $f$ is a (weak) track on $G$ meaning that there exists a non-negative integer $\ell$ such that $f$ is a (weak) $\ell$-track. Similarly, we say that $F$ is a (weak / patient) tour meaning that there exists a non-negative integer $\ell$ such that $F$ is a (weak / patient) $\ell$-tour.
\end{remark}

\section{The span capacities}\label{sec:results}

We now restate the definitions of the three vertex span variants. These were originally introduced in \cite{BaTa23}, while their track-based equivalents were presented in \cite{Erceg23}. For the first two variants, we adopt the track-based terminology. For the third, we use patient tours instead of \emph{opposite tracks}, as used in \cite{Erceg23}. It is straightforward to verify that these notions are equivalent.

\begin{definition}[\cite{BaTa23, Erceg23}]\label{def:allSpans}
Let $G$ be a connected graph. 
\begin{enumerate}[(i)]
    \item The \emph{strong vertex span} of the graph $G$, denoted by $\svSpan{G}$, is $$\svSpan{G} = \max \{ m_G(f,g)  \mid  \ell\in \mathbb{Z}_{\ge 0} \text{ and } f,g \text{ are weak $\ell$-tracks on $G$} \}.$$
    \item The \emph{direct vertex span} of the graph $G$, denoted by $\dvSpan{G}$, is $$\dvSpan{G} = \max \{ m_G(f,g)  \mid  \ell\in \mathbb{Z}_{\ge 0} \text{ and } f,g \text{ are $\ell$-tracks on $G$} \}.$$
    \item The \emph{Cartesian vertex span} of the graph $G$, denoted by $\cvSpan{G}$, is $$\cvSpan{G} = \max \{ m_G(f,g)  \mid  \ell\in \mathbb{Z}_{\ge 0} \text{ and } \{f,g\} \text{ is a patient $\ell$-tour on $G$} \}.$$
\end{enumerate}
 \end{definition}

Now, we give the definitions of three versions of capacities corresponding to three different versions of spans. 

\begin{definition}\label{def:direct-capacity}
    Let $G$ be a non-trivial graph, and $d \leq \sigma^\times_V(G)$ a natural number. We define the \emph{direct $d$-capacity} of the graph $G$, denoted by $\dcap{d}{G}$, as the maximum integer $c$ such that there exists an 
    $\ell$-tour $F=\{f_1,f_2,\ldots,f_c\}$ on $G$ satisfying $m_G(F)=d$.
    If $d=\sigma^\times_V(G)$, we call it the \emph{direct capacity} and denote it by $\dcap{}{G}$.
\end{definition}

\begin{definition}\label{def:strong-capacity}
    Let $G$ be a non-trivial graph, and $d \leq \sigma^\boxtimes_V(G)$ a natural number. We define the \emph{strong $d$-capacity} of graph $G$, denoted by $\scap{d}{G}$, as the maximum integer $c$ such that there exists a weak $\ell$-tour $F=\{f_1,f_2,\ldots,f_c\}$ on $G$ satisfying $m_G(F)= d$.
    If $d=\sigma^\boxtimes_V(G)$, we call it the \emph{strong capacity} and denote it by $\scap{}{G}$.
\end{definition}

\begin{definition}\label{def:Cartesian-capacity}
    Let $G$ be a non-trivial graph different from a path and $d \leq \sigma^\square_V(G)$ a natural number. We define the \emph{Cartesian $d$-capacity} of graph $G$, denoted by $\ccap{d}{G}$, as the maximum integer $c$ such that there exists a patient $\ell$-tour $F=\{f_1,f_2,\ldots,f_c\}$ on $G$ satisfying $m_G(F)= d$.
    If $d=\sigma^\square_V(G)$, we call it the \emph{Cartesian capacity} and denote it by $\ccap{}{G}$.
\end{definition}

Note that in Definitions \ref{def:direct-capacity} and \ref{def:strong-capacity} we require the graphs to be non-trivial and in Definition \ref{def:Cartesian-capacity} not a path, as well. This is due to the fact that the corresponding spans in these cases are equal to 0. If we were to allow for $d$ to equal zero in the above definitions, those corresponding capacities would be unbounded and of no interest for research. Hence, in the above definitions all observed graphs have the corresponding span at least 1, so $d$ being a natural number less than or equal to the corresponding span is well-defined. 

\subsection{Basic results}

Some basic relations regarding capacities are given in following propositions.  

\begin{proposition}\label{prop:basic1a}
    Let $G$ be a non-trivial graph. For each $d\in\mathbb{N}$, with $d\leq \sigma^\times_V(G)$, it holds true that
    $$\scap{d}{G}\geq \dcap{d}{G}.$$
\end{proposition}
\begin{proof}
    The inequality follows immediately from the fact that every $\ell$-track is also a weak $\ell$-track.
\end{proof}

\begin{proposition}\label{prop:basic1b}
    Let $G$ be a non-trivial graph different from a path. For each $d\in\mathbb{N}$, with $d \leq \sigma^\square_V(G)$, it holds true that
    $$\scap{d}{G}\geq \ccap{d}{G}.$$
\end{proposition}
\begin{proof}
    The inequality follows immediately from the fact that every patient $\ell$-track is also a weak $\ell$-track.
\end{proof}

\begin{proposition}\label{prop:basic2a}
    Let $G$ be a non-trivial graph. For each $\mathbin{\ast} \in \{\boxtimes,\times\}$, and for each $d\in\mathbb{N}$, with $d\leq\sigma^\ast_V(G)$, it holds true that
    $$\mathrm{cap}_{d}^{\ast}(G)\geq \mathrm{cap}^{\ast}(G)\geq 2.$$
\end{proposition}

\begin{proof}
    This claim is easily seen directly from the definition of the direct and strong vertex span.
\end{proof}

\begin{proposition}\label{prop:basic2b}
    Let $G$ be a non-trivial graph different from a path. For each $d\in\mathbb{N}$, with $d\leq\sigma^\square_V(G)$, it holds true that
    $$\ccap{d}{G}\geq \ccap{}{G}\geq 2.$$
\end{proposition}

\begin{proof}
    This claim is easily seen directly from the definition of Cartesian vertex span.
\end{proof}

It is easy to see that there exist graphs for which any variant of capacity is greater than $2$. The smallest (regarding order and size) such example, in the case of the strong capacity, is $P_3$. We have $\sigma^\boxtimes_V(P_3)=1$. It holds true that
$$\scap{}{P_3}=3.$$

\subsection{Paths and cycles}

We will start by observing capacities of well-known classes of graphs: paths and cycles. It is known from \cite{BaTa23} and \cite{Erceg23}, that for any $n\geq2$, $\svSpan{P_n}=\dvSpan{P_n}=1$ and $\cvSpan{P_n}=0$. For cycles, for any $n\geq 3$, it holds true that $\svSpan{C_n}=\dvSpan{C_n}=\lfloor\frac{n}{2}\rfloor$ and 
$\cvSpan{C_n}=\begin{cases}
 \lfloor\frac{n}{2}\rfloor, & n \text{ is odd,} \\
     \frac{n}{2}-1, & n \text{ is even.}
\end{cases}$

\begin{proposition}\label{prop:direct-path}
    For $n\geq 2$, $\dcap{}{P_n}=2$. 
\end{proposition}
\begin{proof}
    The claim is obvious for $n=2$.
    Let $n>2$. From Proposition \ref{prop:basic2a} we have $\dcap{}{P_n}\geq 2$. It holds true that $\sigma^\times_V(P_n)=1$. Let us assume that $\dcap{}{P_n}>2$ and denote $k=\dcap{}{P_n}$.    
    Let $F=\{f_1,\ldots,f_k\}$ be a collision-free $\ell$-tour on $P_n$. Since $k\geq 3$ there are at least two $\ell$-tracks, say $f,g\in\{f_1,\ldots,f_k\}$, such that $d(f(0),g(0))$ is even. Let the vertices of $P_n$ be denoted by $1,\ldots,n$ in a natural order. Let $f(0)=i$ and $g(0)=j$, for some $i,j\in\{1,\ldots,n\}$, and without the loss of generality assume that $i<j$. By active movement rules, the distance between $f$ and $g$ in each step will be even. Since $f$ and $g$ are $\ell$-tracks, there exists a $t\in\{1,\ldots,\ell\}$ such that $f(t)=n$, and since $d(f(t),g(t))\geq1$ it follows that $g(t)<f(t)$. That means that there is some $t'\in\{1,\ldots,t-1\}$ such that $f(t')=g(t')$. That is a contradiction with $m_{P_n}(f,g)\geq 1$.
\end{proof}

 The strong capacity of paths equals to $n$, as follows from Proposition \ref{prop:strong-1}, which is stated for a larger class of graphs, and the Cartesian capacity is not observed for paths, since $\sigma^\square_V(P_n)=0$. Therefore, all possible capacity values for paths are proven.

For cycles we have the following results.

\begin{proposition}\label{prop:cyc_sd}  Let $n,d\in\mathbb{N}$ such that $n\geq 3$ and $1\leq d\leq \sigma^\boxtimes_V(C_n)=\sigma^\times_V(C_n)$. 
    It holds true that 
    $$\scap{d}{C_n}= \dcap{d}{C_n}=\left\lfloor\frac{n}{d}\right\rfloor.$$
    
\end{proposition}
\begin{proof} Firstly, we will show that $\scap{d}{C_n} \leq\left\lfloor\frac{n}{d}\right\rfloor$, bounding the strong and direct capacity, due to Proposition \ref{prop:basic1a}. 
    Towards a contradiction, suppose $\scap{d}{C_n}=k \geq \left\lfloor\frac{n}{d}\right\rfloor+1$. Without loss of generality let us denote the vertices of $C_n$ by $0,\ldots,n-1$ in a natural order and let us observe a weak $\ell$-tour $F=\{f_1,\ldots,f_k\}$ such that  $m_{C_n}(F)=d$ and such that $f_1(0)=0$, $f_i(0)\leq f_j(0)$, for all applicable $i<j$. Since the distance between any two weak $\ell$-tracks of $F$ is at least $d$, we have at least $kd$ edges in the graph $C_n$. Hence
    $$kd \geq \left(\left\lfloor\frac{n}{d}\right\rfloor+1\right) d >n,$$ so we arrive at the conclusion that $C_n$ has more than $n$ edges, which is a contradiction.
    
    Next, for $k=\left\lfloor\frac{n}{d}\right\rfloor$, we construct an $(n-1)$-tour $F=\{f_1,...,f_k\}$, with $m_{C_n}(F)=d$, consequently proving
    the stated direct capacity value.
    
    For each $t\in\{0,\ldots,n-1\}$, $j\in\{1,\ldots,k\}$ let
    $$f_j(t)=t+(j-1)d \pmod n.$$
      It is easy to see that $F=\{f_1,\ldots,f_k\}$ is an $(n-1)$-tour, and that $m_{C_n}(F)= d$.
\end{proof}

\begin{proposition}\label{prop:cyc_c}  Let $n,d\in\mathbb{N}$ be such that $n\geq 3$ and $1\leq d\leq \sigma^\square_V(C_n)$. 
    It holds true that 
    $$\ccap{d}{C_n} =\left\{
	\begin{array}{ll}
		\left\lfloor\frac{n}{d}\right\rfloor,  & n \not\equiv 0 \pmod d \\
        \\
		\frac{n}{d} -1, & n \equiv 0 \pmod d 
	\end{array}
\right.$$
    
\end{proposition}

\begin{proof}
    
From Propositions \ref{prop:basic1b} and \ref{prop:cyc_sd} we know that $\ccap{d}{C_n}\leq \scap{d}{C_n}=\left\lfloor\frac{n}{d}\right\rfloor$. 
First we prove that $\ccap{d}{C_n}\neq\left\lfloor\frac{n}{d}\right\rfloor$ when $n \equiv 0 \pmod d$. Let the vertices of $C_n$ be denoted by $0,...,n-1$ in a natural way. Assume the contrary, and let $F=\{f_1,\ldots,f_k\}$, $k=\left\lfloor\frac{n}{d}\right\rfloor=\frac{n}{d}$ be a patient $\ell$-tour in $C_n$ such that $m_{C_n}(F)= d$. Without loss of generality assume $f_j(0)=(j-1)d$, for all $j\in\{1,\ldots,k\}$. 
    
    Let us say that any two patient $\ell$-tracks on the distance $d$ in the starting position, are neighbouring $\ell$-tracks. It is clear that for $j\in\{1,\ldots,k\}$, for which $f_j(1)\neq f_j(0)$, the distance to one of its neighbouring $\ell$-tracks is $d-1$, which is a contradiction.\\

Let us now construct a patient $(n-1)k$-tour that proves the claim.
Let $$k=\left\{
	\begin{array}{ll}
		\left\lfloor\frac{n}{d}\right\rfloor,  & n \not\equiv 0 \pmod d \\
        \\
		\frac{n}{d} -1, & n \equiv 0 \pmod d. 
	\end{array}
\right. $$
We define
         
    $$f_j(t)=(j-1)d+\left\lfloor\frac{t+j-1}{k}\right\rfloor \pmod n,$$  for all $j\in\{1,\ldots,k\}$,  $t\in\{0,\ldots,(n-1)k\}$. Now $F=\{f_1,\ldots,f_k\}$ is a patient $(n-1)k$-tour, and $m_{C_n}(F)=d$.\\
    
    The claim for Cartesian capacity value now follows.
\end{proof} 

\subsection{Bipartite graphs}

Bipartite graphs are a large class of graphs and their span values vary. Some results on selected  families of bipartite graphs can be found in \cite{Erceg23, SuVo24}. In this section we present some results on $2$-capacities of bipartite graphs, results on $1$-capacities are presented in Section \ref{sec:topfull}. 

\begin{proposition}\label{prop:bipartite-Hamiltonian}
Let $G$ be a Hamiltonian bipartite graph $G=(X,Y, E)$ with $n$ vertices and $|X|=|Y|$. It holds true that $\dcap{2}{G} =\scap{2}{G} =|X|=\frac{n}{2}$.
\end{proposition}

\begin{proof}
Let $x_1 \ldots x_n x_1$ be a Hamiltonian cycle in $G$ and let $f_i(0)=x_{2i-1}$ for each $i\in \{1,\ldots,|X|\}$. For all $i\in \{1,\ldots,|X|\}$ and all $t \in \{0,\ldots,n-1\}$ we inductively put $$f_i(t)=x_k \implies f_i(t+1)=
\begin{cases}
    x_{k+1},  & \mbox{if } k \neq n, \\
		x_1,  & \mbox{if } k=n.
\end{cases}
$$ 

It is easy to see that $F=\{f_1,\ldots,f_{|X|}\}$ is an weak $(n-1)$-tour on $G$, such that $m_G(F) =2$. 
\end{proof}

\begin{proposition}\label{prop:bipartite-direct-2}
   Let $G$ be a connected bipartite graph $G=(X,Y,E)$, such that $2\leq |X| \leq |Y|$. It holds true that $\dcap{2}{G}\leq|X|$. If $G$ is a complete bipartite graph, then $$\dcap{}{G}=\scap{}{G}=|X|.$$
\end{proposition}
\begin{proof}
Let $G$ be a connected bipartite graph $G=(X,Y)$, such that $2\leq|X| \leq |Y|$. Towards a contradiction, assume $\dcap{2}{G}>|X|$. For the purpose of this proof, let us denote $\dcap{2}{G}=k$. Let $F=\{f_1,\ldots,f_k\}$ be an $\ell$-tour on $G$ such that $m_G(F)\geq 2$. For some $t\in\{0,...,\ell\}$ let there be $x$ $\ell$-tracks in vertices of $X$ and the rest in vertices of $Y$, namely, $X_1=\{f(t)\in X:f\in F\}\subseteq X$, $|X_1|=x$ and $Y_1=\{f(t)\in Y:f\in F\}\subseteq Y$, $|Y_1|=y$. Obviously $x+y=k>|X|$, so the set $Y_1$ is non-empty, hence $y\geq 1$. Since distance between any two $\ell$-tracks at any point is at least $2$, we may conclude that no vertex from $Y_1$ is adjacent to any in $X_1$, $N(Y_1)\cap X_1=\emptyset$. This means that $N(Y_1)\subseteq X\setminus X_1$, and $|N(Y_1)|\leq |X|-x$. However, due to the active movement rules, for all the $\ell$-tracks in the set $\{f\in F:f(t)\in Y\}$, $f(t+1)$ is in $X$, more precisely, it is in $N(Y_1)$. Since $|N(Y_1)|\leq |X|-x<y$, there will have to exist some $f,g\in F$ such that $f(t+1)=g(t+1)$, which is a contradiction with the assumption of the distance at least $2$ between each two $\ell$-tracks.\\

For the other part of the claim, let $G=K_{r,s}$ be a complete bipartite graph and $2\leq r \leq s$. It is easy to see that $\scap{}{G} =\scap{2}{G} \leq r$, since the graph is complete and $d=\sigma_V^{\boxtimes}(G)=2$. Let us construct a $(2s-1)$-tour $F=\{f_1,\ldots,f_r\}$ such that $m_{G}(F)= 2$ . Denote vertices in $X$ by $x_0,\ldots,x_{r-1}$ and vertices in $Y$ by $y_0,\ldots,y_{s-1}$. For each $j\in\{1,\ldots,r\}$ and each $t\in\{0,\ldots,2s-1\}$ we define

$$f_j(t)=\left\{
	\begin{array}{ll}
		x_{j-1+\frac{t}{2} \pmod r}  & \mbox{, for even } t \\
		y_{j-1+\frac{t-1}{2} \pmod s} & \mbox{, for odd } t
	\end{array}
    \right.$$

Functions defined in this way have the desired properties, they traverse all of the vertices and $m_{G}(F)= 2$, so $\dcap{}{G}=r=|X|$. 

The claim $cap^{\boxtimes}(G)=|X|$ now follows from Proposition \ref{prop:basic1a}.
\end{proof}

\begin{remark}
    For a complete bipartite graph $G=(X,Y)$, $2\leq |X| \leq |Y|$, it holds true that $\sigma_V^{\square}(G)=1$ and for its Cartesian capacity we have $\ccap{}{G}=n-1$, as proved in Proposition \ref{thm:cartesian-topfull}.
\end{remark}

\subsection{Topfull graphs}\label{sec:topfull}

In this section we characterise graphs in which capacity reaches its theoretical maximum. Such graphs are called topfull and defined as follows.

\begin{definition}
We say that a connected graph $G$ on $n$ vertices is 
\begin{enumerate}[(i)]
    \item \emph{strong-topfull} if $\scap{1}{G}=n$, 
    \item \emph{direct-topfull} if $\dcap{1}{G}=n$, and
    \item \emph{Cartesian-topfull} if $\ccap{1}{G}=n-1$.
\end{enumerate}
\end{definition}

Note that if $G$ is a connected graph on $n$ vertices, then by Definition \ref{def:patient-march} $\ccap{1}{G}<n$. Hence, the definition of Cartesian-topfull above.

\begin{proposition}\label{prop:strong-1}
    Any connected graph $G$ is strong-topfull.
\end{proposition}

\begin{proof}
   Suppose $G$ is a connected graph on $n$ vertices, say $v_1,\ldots,v_n$. We need to show that we can construct a collision-free weak $\ell$-tour $F=\{f_1,\ldots,f_n\}$ on $G$. 

   Let $f^1_i(0)=v_i$, for each $i\in\{1, 2, \ldots, n\}$, and let $f^1_1:\{0, 1, \ldots, \ell_1\}\rightarrow V(G)$ be an arbitrary $\ell_1$-track on $G$ such that $f^1_1(0)=v_1$. Since $G$ is connected, such a track (i.e. a walk through all vertices) clearly exists. For each $t\in\{1,\ldots,\ell_1\}$ and each $j\in\{2,\ldots,n\}$ define
   \[
    f^1_j(t) = \begin{cases}
                f^1_1(t-1), & \text{if } f^1_1(t) = f^1_j(t-1) \\
                f^1_j(t-1), & \text{otherwise}
             \end{cases}\text{.}
   \]
    It follows that $F^1=\{f^1_1, \ldots, f^1_n\}$ is a collision-free weak $\ell_1$-march on $G$ such that $f^1_1$ is an $\ell_1$-track.    

    We define the rest inductively. Assume $1<k\leq n$ and that $F^{k-1}$ is a collision-free weak $\ell_{k-1}$-march on $G$ such that $f^{k-1}_{k-1}$ is an $\ell_{k-1}$-track on $G$.
    Let $f^k_k:\{0, 1, \ldots, \ell_k\}\rightarrow V(G)$ be an arbitrary $\ell_k$-track on $G$ such that $f^k_k(0)=f^{k-1}_k(\ell_{k-1})$, and let $f^k_i(0)=f^{k-1}_i(\ell_{k-1})$, for all $i\in\{1,\ldots,n\}$. For each $t\in\{1,\ldots,\ell_k\}$ and each $j\in\{1,\ldots,k-1,k+1,\ldots,n\}$ define
   \[
    f^k_j(t) = \begin{cases}
                f^k_k(t-1), & \text{if } f^k_k(t) = f^k_j(t-1) \\
                f^k_j(t-1), & \text{otherwise}
             \end{cases}\text{.}
   \]
   Again, $F^k=\{f^k_1, \ldots, f^k_n\}$ is a collision-free weak $\ell_k$-march on $G$ such that $f^k_k$ is an $\ell_k$-track.

   It follows that $F:=F^n F^{n-1} \ldots F^1 = \{f_1, f_2, \ldots, f_n\}$ is a collision-free weak $\ell$-march on $G$, where $\ell=\ell_1+\ell_2+\ldots +\ell_n$. Moreover, $f_i$ is a track on $G$, for each $i\in \{1,\ldots,n\}$. Hence, $F$ is a collision-free weak $\ell$-tour on $G$.
\end{proof}

\begin{proposition}
    Let $G$ be a connected graph with $n\geq 3$ vertices. If $G$ has a leaf then $G$ is not direct-topfull. 
\end{proposition}
\begin{proof}
    Let $v$ be a leaf in $G$ and $u$ its only neighbour. Towards a contradiction, assume $G$ is direct-topfull. Hence there exists a collision-free $\ell$-tour $F=\{f_1,\ldots,f_n\}$ on $G$. Let $i$ and $j$ be such that $f_i(0)=v$ and $f_j(0)=u$. Since $f_i$ is an $\ell$-track, it follows that $f_i(1)=u$. Since $F$ is collision-free it follows that $f_k(1)=v$, for some $k\in\{1,\ldots,n\}$. Moreover, $u$ is the only neighbour of $v$ and $f_j(0)=u$, therefore $k=j$. Similarly, it follows that $f_i(t) = f_j(t-1)$ and $f_j(t) = f_i(t-1)$ for every $t\in\{2,\ldots, \ell\}$. Hence, $\orb(F,f_i(0))=\{u,v\}\not= V(G)$, a contradiction with the assumption that $F$ is a tour.
\end{proof}

\begin{definition}
A vertex cycle/edge cover of a graph $G$ is a set $S$ of subgraphs of $G$ such that each subgraph in $S$ is either a cycle or isomorphic to $K_2$ and every vertex of $G$ belongs to a subgraph in $S$. If any two distinct subgraphs in $S$ have no vertices in common, then $S$ is called a vertex disjoint cycle/edge cover of $G$. 
\end{definition}

Note, if $S$ is a vertex-disjoint cycle/edge cover of a graph $G$, the union of the subgraphs in $S$ forms a spanning subgraph originally defined by Tutte as a $Q$-factor \cite{tutte1953}. 

\begin{lemma}\label{lem:cover-track}
Let $G$ be a connected graph on $n$ vertices such that $G$ has a vertex disjoint cycle/edge cover. For every natural number $\ell$ there exists a collision-free $\ell$-march $F=\{f_1, \ldots, f_n\}$ on $G$.
\end{lemma}

\begin{proof}
Let $\ell$ be arbitrary, $V(G)=\{v_1, \ldots, v_n\}$ and $S=\{S_1, \ldots, S_k\}$ be a vertex disjoint cycle/edge cover of $G$. For each $i\in \{1, 2, \ldots, n\}$ we define $f_i : \{0, 1, \ldots, \ell\}\rightarrow V(G)$ as follows. First, $f_i(0)=v_i$, for each $i\in\{1, 2, \ldots, n\}$. Let $0<t\leq \ell$ and assume that $f_i(t-1)$, for each $i\in\mathbb{N}$, are well-defined. Since $S$ is a vertex disjoint cycle/edge cover of $G$, for every $i\in\{1, 2, \ldots, n\}$ there exists exactly one $j \in \{1, 2, \ldots, k\}$ such that $f_i(t-1) \in S_j$. Depending on the structure of $S_j$ we define $f_i(t)$, for all $i\in\{1, 2, \ldots, n\}$.

If $S_j$ is isomorphic to $K_2$, say with vertices $f_{i_1}(t-1)$ and $f_{i_2}(t-1)$, define $f_{i_1}(t) = f_{i_2}(t-1)$ and $f_{i_2}(t) = f_{i_1}(t-1)$. 

If $S_j$ is a cycle of length $c$, say $f_{i_1}(t-1) f_{i_2}(t-1) \ldots f_{i_c}(t-1) f_{i_1}(t-1)$, define $f_{i_1}(t) = f_{i_2}(t-1), f_{i_2}(t) = f_{i_3}(t-1),\ldots, f_{i_{c-1}}(t) = f_{i_c}(t-1), f_{i_c}(t) = f_{i_1}(t-1)$. 

Since $S$ is a vertex cover all $f_i(t)$ have been defined. 

From our definition of $f_i$'s, it is clear that for any $i\in \{1, 2, \ldots, n\}$ and any $t\in \{1, 2, \ldots, \ell-1\}$ it holds true that $f_i(t)f_i(t+1) \in E(G)$. The fact that $f_i(t) \not= f_{i'}(t)$, for each $t \in \{0, 1, \ldots, \ell\}$ and for each two distinct $i, i' \in \{1, 2, \ldots, n\}$ follows from the fact that $S$ is a vertex-disjoint cycle/edge cover. Therefore $F:=\{f_1, \ldots, f_n\}$ is a collision-free $\ell$-march on $G$.
\end{proof}

For a connected graph $G$ and a vertex disjoint cycle/edge cover of $G$, say $S$, the march obtained by Lemma \ref{lem:cover-track} is called \emph{the march induced by $S$}.

\begin{lemma}\label{lemma:fs-give-ce-cover}
Let $G$ be a connected graph on $n$ vertices. If there exists a collision-free $\ell$-march  $F=\{f_1, \ldots, f_n\}$ on $G$, for some natural number $\ell$, then there exists a vertex-disjoint cycle/edge cover of $G$.
\end{lemma}

\begin{proof}
    First notice that since $F$ is an $\ell$-march, for every $t\in \{0, 1, \ldots, \ell-1\}$ and every $i\in\{1,\ldots, n\}$ it holds true that $f_i$ is a walk, hence
    \begin{equation}\label{cond:f_i-march}
      f_i(t+1)\neq f_i(t).
    \end{equation}
    Next, since $F$ is collision-free, for every $t\in \{0, 1, \ldots, \ell\}$, every $i\in\{1,\ldots, n\}$ and every $j\in\{1,\ldots,i-1,i+1,\ldots,n\}$ it holds true that $f_i(t)\neq f_j(t)$. Since the cardinality of $F$ is $n$ and there are $n$ vertices in $G$, the above properties imply that for each $t\in \{0, 1, \ldots, \ell\}$ we have 
    \begin{equation}\label{cond:f_i-everybody}
    V(G)=F(t). 
    \end{equation}

    Let us observe the walks $f_1, \ldots, f_n$ in one time step, say from $t$ to $t+1$. Denote the vertices of $G$ by $v_i:=f_i(t)$ for all $i\in\{1,\ldots, n\}$. Depending on the 'type of move' in step $t\rightarrow t+1$ (there are two possibilities) we define the set $S_t$ of subgraphs of $G$ by the following algorithm. First let $S_t=\emptyset$.
    
    For every two distinct $i_1,i_2\in \{1,\ldots, n\}$ such that $f_{i_1}(t+1)=f_{i_2}(t)$ and $f_{i_2}(t+1)=f_{i_1}(t)$ add the subgraph induced by $\{v_{i_1},v_{i_2}\}$ (it is isomorphic to $K_2$), to the set $S_t$. Informally, players on the vertices $v_{i_1}$ and $v_{i_2}$ exchange positions.

    Next, we claim that for every $i_1,i_2 \in \{1,\ldots, n\}$ satisfying $f_{i_1}(t+1)=f_{i_2}(t)$ and $f_{i_2}(t+1)\neq f_{i_1}(t)$ there is a cycle $v_{i_1}v_{i_2}\ldots v_{i_c}v_{i_1}$ on $c>2$ vertices representing the move of players on these vertices in the same direction along this cycle in the observed step. First, by condition \eqref{cond:f_i-everybody} we know that there exists $j\in \{1,\ldots, n\}$, $j\not\in\{i_1,i_2\}$, such that $f_{j}(t+1)=f_{i_1}(t)$ (somebody moves to $v_{i_1}$). Since $f_{i_2}(t+1)\neq f_{i_1}(t)$ using \eqref{cond:f_i-march} we know that there exists $i_3\in \{1,\ldots, n\}\setminus\{i_1,i_2\}$ such that $f_{i_2}(t+1)=f_{i_3}(t)$. If $i_3=j$, then we get the cycle $v_{i_1}v_{i_2}v_{i_3}v_{i_1}$. If $i_3\neq j$ then, again by \eqref{cond:f_i-march}, there exists $i_4\in \{1,\ldots, n\}\setminus\{i_1,i_2,i_3\}$ such that $f_{i_3}(t+1)=f_{i_4}(t)$. If $i_4=j$ we get the cycle $v_{i_1}v_{i_2}v_{i_3}v_{i_4}v_{i_1}$. If $i_4\neq j$ we continue with analogous observations. Since the set of vertices is finite this process will always end with a cycle $v_{i_1}v_{i_2}\ldots v_{i_c}v_{i_1}$ on $c>2$ vertices.

    For all $i_1,i_2 \in \{1,\ldots, n\}$ satisfying $f_{i_1}(t+1)=f_{i_2}(t)$ and $f_{i_2}(t+1)\neq f_{i_1}(t)$ we add the obtained (mutually disjoint) cycles $v_{i_1}v_{i_2}\ldots v_{i_c}v_{i_1}$, $c>2$, to the set $S_t$. By the above construction it is clear that the obtained set $S_t$ represents a vertex-disjoint cycle/edge cover of $G$ since every vertex $v_i\in V(G)$ is contained in exactly one of the graphs in $S_t$. We call $S_t$ the vertex-disjoint cycle/edge cover of $G$ obtained from the march $F$ at time $t$.
\end{proof}

\begin{theorem}\label{thm:direct-topfull}
    A connected graph $G$ is direct-topfull if and only if for any two vertices $u,v
    \in V(G)$ there exists a $u,v$-path $P$ such that every edge of $P$ belongs to a subgraph of some vertex-disjoint cycle/edge cover.
\end{theorem}

\begin{proof} 
Assume that $G$ is a connected direct-topfull graph on $n$ vertices. Hence, there exists a collision-free $\ell$-tour $F=\{f_1, \ldots, f_n\}$ on $G$. Let $u,v\in V(G)$ be arbitrary. There exists $i\in \{1,\ldots, n\}$ such that $f_i(0)=u$. Since $f_i$ is an $\ell$-track and by definition surjective, there exists $t\in \{0, \ldots, \ell\}$ such that $f_i(t)=v$. Moreover, the sequence of vertices $f_i(0),f_i(1),\ldots,f_i(t)$ forms a walk $W$ from $u$ to $v$ in $G$. Therefore, there also exists a $u,v$-path $P$ in $G$ whose edges are a subset of the edges traversed by this walk.

Let $xy \in E(P)$ be an arbitrary edge of the path $P$. Because $xy\in E(W)$, there must exist some specific time step $t' \in \{0, \ldots, t-1\}$ such that $\{f_i(t'), f_i(t'+1)\} = \{x, y\}$. Using the same line of thought as in the proof of Lemma \ref{lemma:fs-give-ce-cover}, one can construct $S_{t'}$, the vertex-disjoint cycle/edge cover of $G$ obtained from the march $F$ at time $t'$. 
By the construction of $S_{t'}$, the edge $xy$ belongs to a subgraph in $S_{t'}$.
Since $xy$ was an arbitrary edge of $P$, it follows that every edge of $P$ belongs to a subgraph of some vertex-disjoint cycle/edge cover. This concludes the first part of the proof.

For the converse, assume that for any two vertices $u,v \in V(G)$ there exists a $u,v$-path $P$ such that every edge of $P$ belongs to a subgraph of some vertex-disjoint cycle/edge cover. We will construct a collision-free $\ell$-tour $F=\{f_1, \ldots,f_n\}$ on $G$.

Let $V(G) = \{v_1, \ldots, v_n\}$. By assumption, $G$ has a vertex-disjoint cycle/edge cover. Therefore, by Lemma \ref{lem:cover-track}, there exists a collision-free $1$-march on $G$. Let $F' = \{f_1, \ldots, f_n\}$ be this initial collision-free march. 

We construct the tour by systematically extending the march so that each walk $f_i$ visits every vertex. For each index $i \in \{1, \ldots, n\}$ and each vertex $x \in V(G)$, we extend $F'$ so that the walk $f_i$ reaches $x$ as follows. Let $\ell'$ be the current length of $F'$, and let $u = f_i(\ell')$. By assumption, there exists a $u,x$-path $P$ with vertices $u = p_0, p_1, \ldots, p_k = x$ such that each edge $p_j p_{j+1}$ belongs to a subgraph of some vertex-disjoint cycle/edge cover $S_j$.

We traverse $P$ edge by edge. For each $j \in \{0, \ldots, k-1\}$, let $S_j$ be the vertex-disjoint cycle/edge cover containing the edge $p_j p_{j+1}$. By Lemma \ref{lem:cover-track}, $S_j$ induces a collision-free $1$-march, say $M_j = \{g_1, \ldots, g_n\}$. If the component of $S_j$ containing $p_j p_{j+1}$ is a cycle, we choose the orientation of the walks on that cycle defined in Lemma \ref{lem:cover-track} such that the function $g_s$ satisfying $g_s(0) = p_j$ has $g_s(1) = p_{j+1}$. We extend $F'$ with $M_j$ (i.e., we update $F'$ to the extension $M_j F'$, $F':=M_j F'$). After $k$ such extensions, the new march ends at a step $\ell'+k$ where $f_i(\ell'+k) = x$. 

Repeating this extension process for every index $i$ and every target vertex yields a finite sequence of extensions. Let $F = \{f_1, \ldots, f_n\}$ be the final march obtained after all extensions are complete. By Remark \ref{rem:extension-cf}, extending a collision-free march by another collision-free march preserves the collision-free property, so $F$ is a collision-free march. Furthermore, because our process explicitly extended every walk $f_i$ to include every vertex $x \in V(G)$ in its image, each walk $f_i$ in $F$ is surjective. Thus, the walks are $\ell$-tracks, and $F$ is a collision-free $\ell$-tour, which implies $G$ is direct-topfull.
\end{proof}

The characterization of direct-topfull graphs in Theorem \ref{thm:direct-topfull} naturally connects the direct $1$-capacity of a graph to classic graph factorizations. A \emph{matching} in a graph $G$ is a set of pairwise disjoint edges. A \emph{perfect matching} is a matching that covers every vertex of $G$. A \emph{$1$-factor} of $G$ is a spanning subgraph in which every vertex has degree exactly $1$. Note that the edges of a perfect matching naturally induce a $1$-factor whose connected components are all individually isomorphic to $K_2$. A $2$-factor is a spanning subgraph in which every vertex has degree $2$ (a set of vertex-disjoint cycles). Therefore, any $1$-factor or $2$-factor of $G$ is a vertex-disjoint cycle/edge cover of $G$. This observation yields several immediate consequences for well-studied graph families. 

\begin{corollary}
If $G$ is a Hamiltonian graph, then $G$ is direct-topfull.
\end{corollary}

\begin{proof}
Let $C$ be a Hamiltonian cycle of $G$. Since $C$ is a spanning subgraph of $G$, it is a $2$-factor, and thus a vertex-disjoint cycle/edge cover of $G$. For any arbitrary $u,v \in V(G)$, there exists a $u,v$-path $P$ that is a subgraph of $C$. Consequently, every edge of $P$ belongs to a vertex-disjoint cycle/edge cover, so $G$ is direct-topfull by Theorem \ref{thm:direct-topfull}.
\end{proof}

Remember, a connected graph $G$ is called \emph{bridgeless} if the removal of any single edge from $G$ leaves the graph connected. A graph is \emph{$k$-regular} if every vertex is of degree $k$. A $3$-regular graph is also called a \emph{cubic} graph. The following Petersen's theorem is well-known.

\begin{theorem}\cite{petersen1891theorie,west2001introduction}\label{thm:petersen-cubic}
Every bridgeless cubic graph contains a $1$-factor.
\end{theorem}

\begin{corollary}
If $G$ is a connected bridgeless cubic graph, then $G$ is direct-topfull.
\end{corollary}
\begin{proof}
Let $G$ be a connected bridgeless cubic graph. By Theorem \ref{thm:petersen-cubic}, $G$ contains a $1$-factor, say $M_1$. Since $G$ is a cubic graph and for every vertex $v\in V(G)$, $M_1$ contains exactly one edge incident with $v$, removing the edges of $M_1$ from $G$ leaves a spanning subgraph where every vertex has degree of exactly $2$. By definition, this remaining spanning subgraph is a $2$-factor, say $M_2$. Thus, the edge set of $G$ can be partitioned into $M_1$ and $M_2$. Both $M_1$ and $M_2$ are vertex-disjoint cycle/edge covers of $G$. Because every edge of $G$ belongs to either $M_1$ or $M_2$, every edge of $G$ belongs to at least one vertex-disjoint cycle/edge cover. 

Let $u,v \in V(G)$ be arbitrary. Because $G$ is connected, there exists a $u,v$-path $P$ in $G$. Since every edge of $G$ is part of a vertex-disjoint cycle/edge cover, every edge of $P$ must also belong to some vertex-disjoint cycle/edge cover. The assertion then follows as an immediate consequence of Theorem \ref{thm:direct-topfull}.
\end{proof}

For the next corollary we will use the following Petersen's theorem.

\begin{theorem}\cite{petersen1891theorie,west2001introduction}\label{thm:petersen-2factor}
Every regular graph of positive even degree has a $2$-factor.
\end{theorem}

\begin{corollary}
If $G$ is a connected $2k$-regular graph, for any integer $k \geq 1$, then $G$ is direct-topfull.
\end{corollary}
\begin{proof}
By Theorem \ref{thm:petersen-2factor}, if $k\geq 1$ is an integer, then every $2k$-regular graph has a $2$-factor. By repeatedly extracting a $2$-factor and noting that the remaining graph is $(2k-2)$-regular, $G$ can be decomposed into exactly $k$ edge-disjoint $2$-factors. Since their union spans the entire edge set of $G$, every edge of $G$ belongs to a $2$-factor (and thus to a vertex-disjoint cycle/edge cover). As $G$ is connected, the condition of Theorem \ref{thm:direct-topfull} is satisfied for any path between any two vertices.
\end{proof}

A connected graph is called \emph{elementary} if the union of all its perfect matchings forms a connected subgraph \cite{lovasz2009matching}.

\begin{corollary}
If $G$ is an elementary graph, then $G$ is direct-topfull.
\end{corollary}

\begin{proof}
If $G$ has exactly two vertices, then it is isomorphic to $K_2$, and the assertion follows immediately. 

Suppose $G$ has more than two vertices. By the definition of an elementary graph, the union of all perfect matchings of $G$ forms a connected subgraph, say $H$. Since every perfect matching covers all vertices of $G$, the subgraph $H$ is a spanning subgraph of $G$. Let $u,v \in V(G)$ be arbitrary. Because $H$ is a connected spanning subgraph, there exists a $u,v$-path $P$ in $H$. 

By the construction of $H$, every edge of $P$ belongs to some perfect matching of $G$. The edges of any perfect matching induce a $1$-factor of $G$, which is a spanning subgraph whose connected components are all isomorphic to $K_2$. Therefore, the set of components of this $1$-factor forms a vertex-disjoint cycle/edge cover of $G$. Consequently, every edge of $P$ belongs to a subgraph of some vertex-disjoint cycle/edge cover. The assertion then follows as an immediate consequence of Theorem \ref{thm:direct-topfull}.
\end{proof}

Let $G$ be a connected graph. A vertex $v \in V(G)$ is a \emph{cut vertex} of $G$ if $G-v$ is not connected. We will need the following result by Whitney \cite{Wh01, Wh02}.

\begin{theorem}\cite{Wh01, Wh02}\label{thm:2connected}
Let $G$ be a connected graph with at least three vertices. Then $G$ has no cut vertex if and only if any two distinct vertices lie on a common cycle.  
\end{theorem}

\begin{theorem}\label{thm:cartesian-topfull}
    A connected graph $G$ on  more than three vertices is Cartesian-topfull if and only if it has no cut vertices.
\end{theorem}

\begin{proof}
Suppose that $G$ is Cartesian-topfull. By definition, there exists a patient collision-free $\ell$-tour $F = \{f_1, \ldots, f_{n-1}\}$ on $G$. This implies that for any $t \in \{0, \ldots, \ell\}$, it holds true that $|\{f_1(t), \ldots, f_{n-1}(t)\}|=n-1$ and that at the time $t$ exactly one vertex is vacant. Let $D_F=\{p\}$ be the dual of $F$. By definition, this dual consists of a single track $p: \{0, \ldots, \ell\} \rightarrow V(G)$, which represents the movement of the vacant vertex. For any $t \in \{0, \ldots, \ell-1\}$, the patient property dictates that there is exactly one $j \in \{1, \ldots, n-1\}$ such that $f_j(t) \neq f_j(t+1)$. Furthermore, the collision-free property requires that the destination of this moving track must be the currently vacant vertex; thus, $f_j(t+1) = p(t)$. Consequently, the vertex vacated by $f_j$ becomes the newly vacant vertex at time $t+1$, yielding $p(t+1) = f_j(t)$. Therefore, any movement by a track $f_j$ along an edge precisely corresponds to the dual track $p$ traversing the same edge in the opposite direction.

Towards a contradiction, suppose $c \in V(G)$ is a cut vertex of $G$. Let $G'$ and $G''$ be two distinct components of $G-c$. Without the loss of generality, assume that for some track $f_i \in F$, the initial position is $f_i(0) \in V(G')$. Because $f_i$ is a track, it is surjective and must eventually visit $G''$. Let $t' = \min\{t \mid f_i(t) \in V(G'')\}$. Any path from $G'$ to $G''$ must pass through $c$, hence $f_i(t'-1) = c$. Let $t''=\min\{t \mid f_i(x)=c \text{ for all } t\leq x<t'\}$. Note that $t''$ exists, since $t'-1$ satisfies the condition. Moreover, the choice of $t''$ implies that $f_i(t''-1)\in V(G')$ and $p(t''-1)=c$. Since $f_i(t'')=c$, therefore $p(t'')=f_i(t''-1) \in V(G')$. Moreover, for all $t'' \leq x<t'$ it holds true that $f_i(x)=c$, hence $p(x)\in V(G')$, for all $t'' \leq x<t'$.
Therefore $p(t'-1) \in V(G')$. However, since $f_i$ moves from $c$ to $G''$ at step $t'$, the vacant vertex must be located in $G''$, requiring $p(t'-1) = f_i(t') \in V(G'')$. This yields a contradiction ($V(G') \cap V(G'') = \emptyset$). Hence, a Cartesian-topfull graph contains no cut vertices.

For the converse, assume $G$ has no cut vertices. By Theorem \ref{thm:2connected}, any two distinct vertices in $G$ lie on a common cycle. We construct a patient collision-free $\ell$-tour $F = \{f_1, \ldots, f_{n-1}\}$ on $G$. Let $W = w_0, w_1, \ldots, w_k$ be a spanning walk on $G$. Let $D=\{p\}$ be an initial patient $0$-march with $p(0) = w_0$. Let $F'=\{f'_1, \ldots, f'_{n-1}\}$ be the dual of $D$, hence $\{f'_1(0), \ldots, f'_{n-1}(0)\} = V(G)\setminus \{w_0\}$.

We systematically extend the march $F'$ to guarantee surjectivity for every weak walk in $F'$. Assume $F'$ is currently of length $\ell'$ with the dual $D=\{p\}$ such that $p(\ell')=w_r$. We perform a sequence of extensions of $F'$ such that for each $i \in \{1, \ldots, n-1\}$ the weak walk $f'_i$ visits $w_r$. 

For each $i \in \{1, \ldots, n-1\}$, denote by $u_i = f'_i(\ell')$ (i.e., the current position of $f'_i$). If $u_i = w_r$, the weak walk $f'_i$ has already visited the desired vertex. If $u_i \neq w_r$, by Theorem \ref{thm:2connected}, there exists a cycle $C$ containing both $u_i$ and $w_r$. By Proposition \ref{prop:cyc_c}, cycles are Cartesian-topfull. Therefore, there exists a patient collision-free $m$-tour $F_C=\{g_1,\ldots, g_{|V(C)|-1}\}$ on $C$, for some $m \in \mathbb{N}$. We can choose $F_C$ such that its initial and final occupied vertices are precisely $V(C) \setminus \{w_r\}$; i.e. $F_C(0) = F_C(m) = V(C) \setminus \{w_r\}$.

We extend the tour $F_C$ which is on the cycle $C$, to a patient collision-free $m$-march $F^*$ on the entire graph $G$ as follows. Let $V_{out} = V(G) \setminus V(C)$. For each $v \in V_{out}$, we define a weak walk $s_v$, by $s_v(t) = v$ for all $t \in \{0, \ldots, m\}$, and let $S = \{s_v \mid v \in V_{out}\}$. We then define the march on $G$ as $F^* = F_C \cup S$. 
    
By construction, $F^*$ consists of exactly $n-1$ weak walks. The set of occupied vertices at the beginning of $F^*$ evaluates to $F^*(0) = F_C(0) \cup \{s_v(0) \mid v \in V_{out}\} = (V(C) \setminus \{w_r\}) \cup V_{out} = V(G) \setminus \{w_r\}$. Because the set of occupied vertices at step $\ell'$ of our current march $F'$ is exactly $F'(\ell') = V(G) \setminus \{w_r\}$, the extension condition $F'(\ell') = F^*(0)$ is satisfied. Hence, we can extend $F'$ by appending the march $F^*$, we set $F':=F^* F'$.

Because $F_C$ is a tour on the cycle $C$, every weak walk within $F_C$ is surjective onto $V(C)$. Consequently, the weak walk in $F^*$ extending $f'_i$ traverses all vertices of $C$, meaning it is guaranteed to visit $w_r$. Thus, after updating $F' := F' F^*$, the weak walk $f'_i$ has successfully visited $w_r$. Note that since $F^*(m) = V(G) \setminus \{w_r\}$, the dual weak walk $p$ (the vacant vertex) correctly returns to $w_r$ at the end of this extension.

We sequentially repeat this evaluation and extension process for every index $i \in \{1, \ldots, n-1\}$. Upon completing this inner sequence, every weak walk $f'_i$ in $F'$ has visited the vertex $w_r$, and the vacant vertex remains located at $w_r$.

To continue the construction, we must move the vacant vertex to the next target in the spanning walk $W$, which is $w_{r+1}$. We extend $F'$ by a single patient collision-free step: the unique track currently occupying $w_{r+1}$ moves along the edge $w_{r+1}w_r$ to $w_r$. By the definition of the dual, this single move shifts the position of the dual track $p$ exactly to $w_{r+1}$. 

We iterate this entire procedure over the index $r \in \{0, \ldots, k-1\}$ along the spanning walk $W$. Because $W$ is a spanning walk, every vertex $v \in V(G)$ appears as the target $w_r$ at least once. At each such stop, the march is systematically extended to ensure that every weak walk $f'_i$ visits that specific vertex. Let $F$ be the final march obtained after the walk $W$ is fully traversed. Since the extension sequence explicitly forces every $f'_i \in F$ to visit every vertex in $V(G)$, each weak walk is surjective. Therefore, $F$ is a patient collision-free tour, which concludes that $G$ is Cartesian-topfull.
\end{proof}

\section{Conslusion and open problems}

In this paper, we introduced the notion of span capacities of graphs as a natural generalization of spans of graphs from the two-player setting to configurations involving multiple players maintaining a prescribed minimum distance. We defined three variants of $d$-capacity — strong, direct, and Cartesian — corresponding to the classical movement rules used in graph pursuit and traversal problems. To support these definitions, we developed a unified framework based on marches and tours, which allowed us to formally describe simultaneous movement of multiple players on graphs.

We determined some exact capacity values for several fundamental graph classes, including paths, cycles, complete bipartite graphs and balanced Hamiltonian bipartite graphs, and established general bounds for bipartite graphs for $2$-capacity. A central part of the paper was devoted to the study of topfull graphs, namely graphs attaining the maximum possible $1$-capacity. We proved that every connected graph is strong-topfull, characterized direct-topfull graphs through vertex-disjoint cycle/edge covers, and established that Cartesian-topfull graphs are precisely the $2$-connected graphs. These results reveal strong connections between span capacities and classical structural graph properties such as connectivity, Hamiltonicity, regularity, and graph factorizations.

The concept of graph capacities opens several directions for future research. Possible extensions include studying algorithmic complexity questions, determining capacities for additional graph classes, investigating edge-capacity analogues, and exploring asymptotic behaviour of capacities under graph products and other graph operations. 

In Section \ref{sec:topfull}, for each capacity variant we have characterised graphs for which the $1$-capacity of the variant reaches its theoretical maximum. Solutions to the following two problems would generalise these results.

\begin{problem}
    Given a connected graph $G$ and a positive integer $d$, determine a sharp upper bound for the chosen $d$-capacity variant of $G$. 
\end{problem}

\begin{problem}
    Given a positive integer $d$, characterise graphs $G$ for which the chosen $d$-capacity variant attains its theoretical maximum. 
\end{problem}

Section \ref{sec:results} provides some results on $d$-capacities for specific families of graphs and/or for specific values of $d$. We state the following problem which can be extended to other well-known families of graphs and variants of capacities.

\begin{problem}
    Given a tree $T$ and a positive integer $d$, what is the value of $\ccap{d}{T}$?
\end{problem}

Given a connected graph $G$, its span (any variant) can be determined in polynomial time \cite{BaTa23}. The following problem is then natural to ask.

\begin{problem}
    Given a connected graph $G$ and a positive integer $d$, what is the complexity of determining the chosen $d$-capacity of $G$.
\end{problem}

\section*{Acknowledgments} 
Mateja Gra\v si\v c acknowledges the support of the Slovenian Research and Innovation  Agency (research core funding No. P1-0288). Aljoša Šubašić and Tanja Vojković were partialy supported by the ZMAJ project (IP-UNIST-45) funded by the European Union — NextGenerationEU and the VAL project (PK.3.4.17.0021) funded by the European Regional Development Fund (ERDF). Andrej Taranenko acknowledges the financial support from the Slovenian Research and Innovation Agency (research core funding No. P1-0297, projects N1-0285 and N1-0431).

\bibliographystyle{plain} 
\bibliography{bibliography}

\end{document}